\documentclass{amsart}
\usepackage{latexsym}
\usepackage{amsfonts}
\usepackage{amsmath}
\usepackage{amssymb}
\usepackage{color}

\newcommand{\Zzz}{\hbox{\Cp Z}}

\newcommand{\Rere}{\hbox{\rm Re}}
\newcommand{\Imim}{\hbox{\rm Im}}

\newenvironment{proof_special}{\noindent {\bf Proof:}}{}
\newenvironment{proof_}[1]{\noindent {\bf #1}}{{\qed}}

\newtheorem{theorem}{Theorem}[section]
\newtheorem{proposition}[theorem]{Proposition}
\newtheorem{lemma}[theorem]{Lemma}
\newtheorem{corollary}[theorem]{Corollary}

\newtheorem{definition}[theorem]{Definition}

\makeatletter
\@addtoreset{equation}{section}
\makeatother

\newcommand{\onethingatopanother}[2]
{\genfrac{}{}{0pt}{}{#1}{#2}}

\newcommand{\binomial}[2]
{\genfrac{(}{)}{0pt}{}{#1}{#2}}

\newcommand{\QSym}{{\rm QSym}}
\newcommand{\BQSym}{{\rm BQSym}}

\newcommand{\tensor}{\otimes}

\newcommand{\hz}{\hat{0}}
\newcommand{\ho}{\hat{1}}

\newcommand{\ab}{\av\bv}
\newcommand{\av}{{\bf a}}
\newcommand{\bv}{{\bf b}}
\newcommand{\cd}{\cv\dv}
\newcommand{\ctd}{\cv\mbox{-}2\dv}
\newcommand{\cv}{{\bf c}}
\newcommand{\dv}{{\bf d}}

\newcommand{\imagi}{{\rm i}}

\newcommand{\zab}{\Zzz\langle\av,\bv\rangle}
\newcommand{\zcd}{\Zzz\langle\cv,\dv\rangle}

\font\Cp = msbm10

\begin{document}

\title{Cyclotomic factors of the descent set polynomial}
\author[{\sc D.\ CHEBIKIN},
         {\sc R.\ EHRENBORG},
         {\sc P.\ PYLYAVSKYY} and
         {\sc M.\ READDY}]
        {{\sc DENIS CHEBIKIN},
         {\sc RICHARD EHRENBORG}, \\
         {\sc PAVLO PYLYAVSKYY} and
         {\sc MARGARET READDY}}

\date{\today}

\begin{abstract}

We introduce the notion of the
descent set polynomial as an alternative way of encoding 
the sizes of descent classes of permutations. 
Descent set polynomials exhibit interesting 
factorization patterns. We explore the question 
of when particular cyclotomic factors 
divide these polynomials.
As an instance we deduce that the proportion of odd
entries in the descent set statistics in the symmetric
group~${\mathfrak S}_{n}$ only depends on the number on $1$'s in the
binary expansion of $n$. We observe similar properties for the signed
descent set statistics.

\end{abstract}

\maketitle

\section{Introduction}

The study of the behavior of the descent sets
of permutations
in the symmetric group~${\mathfrak S}_{n}$
on $n$ elements usually involves such
questions as maximizing the
descent set or determining inequalities which hold among the
entries~\cite{de_Bruijn,Ehrenborg_Levin_Readdy,
Ehrenborg_Mahajan,Ehrenborg_Readdy_r,
Niven,Readdy,Sagan_Yeh_Ziegler}.
The usual way to encode the descent statistic information is via
the {\it {Eulerian polynomial}} $A_{n}(t) = \sum_{S} \beta_{n}(S) \cdot
t^{|S|+1}$, where $S$ runs over all subsets of $[n-1] = \{1, \ldots, n-1\}$,
and $\beta_{n}(S)$ denotes the number of permutations of size $n$ with descent
set $S$. We instead introduce the  descent set polynomial  where the
statistic of interest appears in the exponent of the variable
$t$, rather than as a coefficient.  That is, the 
{\em $n$th descent set polynomial}
is defined by
$$ 
     Q_{n}(t) = \sum_{S} t^{\beta_{n}(S)} , 
$$ 
where $S$ ranges over all subsets of $[n-1]$.

The degree of the descent set polynomial is given by the $n$th Euler
number, which grows faster than an exponential. Despite this,
these polynomials appear to have curious factorization
properties, in particular, 
having factors which
are {\it {cyclotomic polynomials}};
see Table~\ref{table_P}. This
paper explains the
occurrence of certain cyclotomic factors.
We have 
displayed these in boldface in the tables.
Both combinatorial 
and number-theoretic properties
(for
example, the number of $1$'s in binary expansion of~$n$
and the prime factorization of~$n$) are involved
in our investigations.

The divisibility by cyclotomic factors is related to the remainders of
sizes of descent classes modulo certain integers. As a simplest example,
$Q_{n}(t)$ is divisible by the second cyclotomic polynomial $\Phi_2$
if and only if the number of even descent set classes is equal to the
number of odd descent set classes. In other words, the proportion of even
and odd entries in the descent set statistics is the same (in the notation
below, $\rho(n) = 1/2$) if and only if $-1$ is a root of the descent set
polynomial. Somewhat surprisingly, whether or not $n$ has this
property depends only on the number of $1$'s in the binary expansion
of $n$.

The paper proceeds as follows.  
In Section~\ref{section_proportion} we look at 
the proportion of odd entries in
the descent set statistics. In Section~\ref{section_QSym} we
discuss this result from the viewpoint of quasisymmetric functions
related to posets. 
We consider similar
properties for the signed descent set statistics
in Section~\ref{section_signed}.
The natural setting for this question is to
look at flag vectors of zonotopes.
In Section~\ref{section_twice_prime} we explore patterns of descent
statistics modulo~$2p$ for prime $p$. 
Here we introduce the descent set
polynomial and consider
divisibility by cyclotomic
polynomials. 
In Section~\ref{section_quadratic}
we explore when the descent set polynomial is divisible by
the quadratic factors $\Phi_2^2$, $\Phi_4^2$ and $\Phi_{2p}^2$. 
In Section~\ref{section_signed_descent} we
introduce type $B$ quasisymmetric functions and 
the {\it {signed descent set
polynomials}}. We use the former to describe divisibility patterns of
the latter. Finally, in 
the concluding remarks
we make a number of observations on the
data presented in Tables~\ref{table_P} and~\ref{table_P_pm}.

\section{The proportion of odd entries}
\label{section_proportion}

\begin{table}[t!]
$$
\begin{array}{c c c c}
k & n = 2^{k} - 1 & \rho(n) & 1/2 - \rho(n) \\ \hline
1 &  1 &  1  & -1/2 \\
2 &  3 & 1/2 &   0  \\
3 &  7 & 1/2 &   0  \\
4 & 15 & 29/2^{6} &   3/2^{6}  \\
5 & 31 & 3991/2^{13} &   3 \cdot 5 \cdot 7 /2^{13}
\end{array}
$$
\caption{The proportion $\rho(n)$ for at most five
$1$'s in the binary expansion of $n$.}
\label{table_one}
\end{table}

For $\pi = \pi_{1} \cdots \pi_{n}$ a permutation in
${\mathfrak S}_{n}$, recall that the {\em descent set} 
of $\pi$ is 
the subset of $[n-1]$ given by
$\{i \: : \: \pi_{i} > \pi_{i+1}\}$.
For a subset $S$ of $[n-1]$
the number of permutations
in~${\mathfrak S}_{n}$ with descent set $S$ is denoted by $\beta_{n}(S)$.

Let $\rho(n)$ denote the proportion of
odd entries in the descent statistics
in the symmetric group~${\mathfrak S}_{n}$, that is,
$$    \rho(n)
 =
   \frac{|\{S \subseteq [n-1] \: : \:
           \beta_{n}(S) \equiv 1 \bmod 2\}|}
       {2^{n-1}}     .   $$
For instance, $\rho(3) = 1/2$ since in the data
$\beta_{3}(\varnothing) = \beta_{3}(\{1,2\}) = 1$
and
$\beta_{3}(\{1\}) = \beta_{3}(\{2\}) = 2$
exactly half of the entries are odd.

The first few values of the proportion $\rho(n)$
are shown in Table~\ref{table_one}. In this section we prove the following
result.
\begin{theorem}
The proportion of odd entries in the descent set statistics
$\rho(n)$ only depend on the number of $1$'s
in the binary expansion of the integer $n$.
\label{theorem_main}
\end{theorem}

Recall a composition of $n$ is a list
$\gamma = (\gamma_{1}, \gamma_{2}, \ldots, \gamma_{m})$
of positive integers such that
$\gamma_{1} + \gamma_{2} + \cdots + \gamma_{m} = n$.
The multinomial coefficient is defined by
$$    \binomial{n}{\gamma}
    =
      \frac{n!}{\gamma_{1}! \cdot \gamma_{2}! \cdots \gamma_{m}!}  .  $$
Define a bijection $D$
between
subsets
of the set $[n-1]$
and compositions of $n$
by sending the set
$\{s_{1} < s_{2} < \cdots < s_{m-1}\}$
to the composition
$(s_{1}, s_{2}-s_{1}, s_{3}-s_{2}, \ldots, n-s_{m-1})$.
Let $\alpha_{n}(S)$ denote the multinomial coefficient
$\binomial{n}{D(S)}$.
The following is a classic result due to MacMahon:
\begin{lemma}
Let $S$ be a subset of $[n-1]$. Then the number of permutations in
$\mathfrak{S}_{n}$ with descent set \emph{contained in $S$} is
$\alpha_{n}(S)$, and we have
$$
\beta_{n}(S) = \sum_{T\subseteq S} (-1)^{|S-T|} \cdot \alpha_{n}(T).
$$
\label{lemma_MacMahon}
\end{lemma}

We need Kummer's theorem for the multinomial coefficient version.
\begin{theorem}
For a prime $p$ and a composition
$\gamma = (\gamma_{1}, \gamma_{2}, \ldots, \gamma_{m})$
of $n$,
the largest power $d$ such that
$p^{d}$ divides the multinomial coefficient $\binomial{n}{\gamma}$
is equal to the number
of carries when adding
$\gamma_{1} + \gamma_{2} + \cdots + \gamma_{m}$ in base $p$.
\end{theorem}
As a corollary we can
determine whether a multinomial coefficient
is even or odd.
This corollary also follows from Lucas' congruence for binomial
coefficients.
\begin{corollary}
For a composition
$\gamma = (\gamma_{1}, \gamma_{2}, \ldots, \gamma_{m})$
of $n$,
the multinomial coefficient $\binomial{n}{\gamma}$ is odd
if and only if 
there are no carries when adding
$\gamma_{1} + \gamma_{2} + \cdots + \gamma_{m}$ in base $2$, that is,
for all $i \neq j$,
the binary expansions of $\gamma_{i}$ and $\gamma_{j}$
have no powers of $2$ in common.
\label{corollary_to_Kummer}
\end{corollary}

Let the binary expansion of $n$ be
$n = 2^{j_{1}} + 2^{j_{2}} + \cdots + 2^{j_{k}}$,
where
$j_{1} > j_{2} > \cdots > j_{k}$.
Call an element of $[n-1]$ \emph{essential}
if it can be expressed as $\sum_{i\in B} 2^{j_i}$ for some
nonempty proper subset $B$ of $[k]$; otherwise, call this element
\emph{nonessential}.

\begin{lemma}
If $S\subseteq [n-1]$ contains a nonessential element $s_{i}$, then
$\alpha_{n}(S)$ is even, that is, $\alpha_{n}(S) \equiv 0 \bmod 2$.
\label{alpha_S_zero_mod_2}
\end{lemma}
\begin{proof}
Let $\gamma = (\gamma_{1}, \gamma_{2}, \ldots, \gamma_{m})$
be the associated composition $D(S)$.
Notice in the addition
$(\gamma_{1} + \cdots + \gamma_{i})
+ 
(\gamma_{i+1} + \cdots + \gamma_{m})
   =
s_{i} + (n-s_{i}) =  n$
there is a carry in base $2$.
Hence it follows from
Corollary~\ref{corollary_to_Kummer}
that $\alpha_{n}(S)$ is even.
\end{proof}

\begin{lemma}
Let $S$ be a subset of $[n-1]$, and suppose that $i \in [n-1] - S$
is an nonessential element. Then
$$
\beta_{n}(S) \equiv \beta_{n}(S \cup \{i\}) \bmod 2.
$$
\label{lemma_nonessential_element}
\end{lemma}
\begin{proof_special}
By Lemmas~\ref{lemma_MacMahon} and~\ref{alpha_S_zero_mod_2}, we have
\begin{eqnarray*}
\hspace*{15mm}
\beta(S\cup \{i\})
  & = &
\sum_{T\subseteq S} (-1)^{|S-T|+1} \cdot \alpha_{n}(T)
 +
\sum_{T\subseteq S} (-1)^{|S-T|} \cdot \alpha_{n}(T\cup\{i\}) \\
  & = &
-\beta_{n}(S) 
 + 
\sum_{T\subseteq S} (-1)^{|S-T|} \cdot \alpha_{n}(T\cup\{i\}) \\
  & \equiv &
\beta_{n}(S) \bmod 2.
\hspace*{15mm}
\hspace*{60mm}\qed
\end{eqnarray*}
\end{proof_special}

\begin{lemma}
Let $S = \{s_{1} < s_{2} < \cdots < s_{m-1}\}$
be a subset of $[n-1]$ consisting
of essential elements, so that there are nonempty proper subsets
$B_{1}, B_{2}, \ldots, B_{m-1}$
of $[k]$ such that $s_{r} = \sum_{i\in B_{r}} 2^{j_{i}}$.
Then $\alpha_{n}(S)$ is odd if and only if 
$B_{1} \subseteq B_{2} \subseteq \cdots \subseteq B_{m-1}$.
\label{lemma_alpha_S_odd}
\end{lemma}
\begin{proof}
Let $B_{0} = \varnothing$ and $B_{m} = [k]$.
If $B_{1} \subseteq B_{2} \subseteq \cdots \subseteq B_{m-1}$
then
$\gamma_{r} = s_{r} - s_{r-1} = \sum_{i\in B_{r}-B_{r-1}} 2^{j_{i}}$.
Then there is no carry in the addition
$\gamma_{1} + \cdots + \gamma_{m} = n$
and hence $\alpha_{n}(S)$ is odd.
On the other hand,
if $\alpha_{n}(S)$ is odd then there is no carries
in the addition $\gamma_{1} + \cdots + \gamma_{m} = n$,
so all the $2$-powers that appears in
$\gamma_{1}, \ldots, \gamma_{m}$ must be disjoint.
Since $s_{r}$ is given by the partial sum
$\gamma_{1} + \cdots + \gamma_{r}$,
the $2$-powers appearing in $s_{r}$ must be
contained among 
the $2$-powers appearing in $s_{r+1}$,
that is, $B_{r} \subseteq B_{r+1}$.
\end{proof}

\begin{lemma}
Let $n$ have $k$ $1$'s in its binary expansion.
Let $E = \{e_{1}, e_{2}, \ldots, e_{2^{k}-1}\}$
be the set of essential elements of
$[n-1]$, where the $e_{i}$'s are listed in increasing order:
$e_{1} < e_{2} < \cdots < e_{2^{k}-1}$.
Let $S = \{s_{i_{1}}, s_{i_{2}}, \ldots, s_{i_{m}}\}$
be a subset of $E$
and $\widehat{S}$ be the set of indices of $S$, that is,
$\widehat{S} = \{i_{1}, i_{1}, \ldots, i_{m}\}$.
Then the parity of $\beta_{n}(S)$ is the same
as the parity of $\beta_{2^{k}-1}(\widehat{S})$.
\label{lemma_beta}
\end{lemma}
\begin{proof_special}
{}From Lemma~\ref{lemma_alpha_S_odd}
it follows that the parity of $\alpha_{n}(S)$ is the same
as the parity of $\alpha_{2^{k}-1}(\widehat{S})$.
Now the result follows by
$$                   \hspace*{10 mm}
      \beta_{n}(S) 
  =
      \sum_{T \subseteq S} (-1)^{|S-T|} \cdot \alpha_{n}(T)
  \equiv
      \sum_{\widehat{T} \subseteq \widehat{S}}
         (-1)^{|\widehat{S}-\widehat{T}|} \cdot \alpha_{2^{k}-1}(\widehat{T})
 =    \beta_{2^{k}-1}(\widehat{S})
         \bmod 2 . 
\hspace*{10 mm} \qed $$
\end{proof_special}

We are now ready to prove Theorem~\ref{theorem_main}.

\begin{proof_}{Proof of Theorem~\ref{theorem_main}:}
It follows from Lemma~\ref{lemma_nonessential_element}
that the proportion of odd entries
among $\beta_{n}(S)$ for $S\subseteq [n-1]$ is the same as the proportion
of odd entries among $\beta_{n}(S)$ for $S\subseteq E$. But by
Lemma~\ref{lemma_beta}, the proportion of odd entries
among $\beta_{n}(S)$ for $S\subseteq E$ depends on $k$, the number of $1$'s
in the binary expansion of $n$.
\end{proof_}

\section{Quasisymmetric functions and posets}
\label{section_QSym}

In this section we relate the preceding result to the theory of
quasisymmetric functions.

Consider the ring $\Zzz[[w_{1},w_{2},\ldots]]$ of
power series with bounded degree.
A function~$f$ in this ring is called {\em quasisymmetric} if for any sequence
of positive integers $\gamma_{1},\gamma_{2},\ldots,\gamma_{m}$ we have
$$
\left[w_{i_1}^{\gamma_{1}}\cdots w_{i_k}^{\gamma_{m}} \right]f =
\left[w_{j_1}^{\gamma_{1}}\cdots w_{j_k}^{\gamma_{m}} \right]f
$$
whenever $i_{1} < \cdots < i_{m}$ and $j_{1} < \cdots < j_{m}$,
and where
$[w^{\gamma}]f$ denotes the coefficient of~$w^{\gamma}$ in $f$. Denote
by $\QSym \subseteq  \Zzz[[w_{1},w_{2},\ldots]]$
the ring of quasisymmetric
functions.

For a composition $\gamma = (\gamma_{1},\gamma_{2},\ldots,\gamma_{m})$
the {\em monomial quasisymmetric function}
$M_{\gamma}$ is given by
$$
M_{\gamma}
 =
\sum_{i_{1} < \cdots < i_{m}}
  w_{i_{1}}^{\gamma_{1}} \cdots w_{i_{m}}^{\gamma_{m}}.
$$

\begin{definition}
Let $P$ be a graded poset with rank function $\rho$.
Define the quasisymmetric function $F(P)$ of the poset $P$
by
$$    F(P)
  =
    \sum_{c} M_{\rho(c)}  ,
$$
where the sum ranges over all chains
$c = \{\hz = x_{0} < x_{1} < \cdots < x_{m} = \ho\}$
in~$P$,
$\rho(c)$ denotes the composition
$(\rho(x_{0},x_{1}), \rho(x_{1},x_{2}), \ldots, \rho(x_{m-1},x_{m}))$,
and $\rho(x,y)$
denotes the rank difference
$\rho(x,y) = \rho(y) - \rho(x)$.
\end{definition}

The quasisymmetric function of a poset is multiplicative,
that is,
for two graded posets $P$ and $Q$ the quasisymmetric
function of their Cartesian product is given by
the product of the respective quasisymmetric functions
$$   F(P \times Q) = F(P) \cdot F(Q)  ;  $$
see Proposition~4.4 in~\cite{Ehrenborg}.
Recall that the Boolean algebra $B_{n}$ is the Cartesian power of the
chain of two elements $B_{1}$, that is, $B_{n} = B_{1}^{n}$
and
$F(B_{1}) = w_{1} + w_{2} + \cdots$.
Hence we have that
$F(B_{n}) = F(B_{1})^{n} 
  = M_{(1)}^{n} = (w_{1} + w_{2} + \cdots)^{n}$.


Let $P$ be a graded poset of rank $n$.
For a subset
$S = \{s_{1} < s_{2} < \cdots < s_{m-1}\}$
of $[n-1]$, define the {\em flag $f$-vector}
entry $f_{S}$
to be the number of chains
$\{\hz = x_{0} < x_{1} < \cdots < x_{m} = \ho\}$
such that
$\rho(x_{i}) = s_{i}$ for $1 \leq i \leq m-1$.
The {\em flag $h$-vector} is defined by the invertible
relation
$$     h_{S}  =  \sum_{T \subseteq S} (-1)^{|S-T|} \cdot f_{T}  . $$

Recall the bijection $D$ between subsets of $[n-1]$
and compositions of $n$ defined in
Section~\ref{section_proportion}.
By abuse of notation
we will write $M_{S}$ instead of $M_{D(S)}$,
where the degree of the quasisymmetric function
is understood.
Then Lemma~4.2 in~\cite{Ehrenborg}
states that the quasisymmetric function of a poset
encodes the flag $f$-vector, that is,
$$
F(P) = \sum_{S \subseteq [n-1]} f_{S} \cdot M_{S}.
$$
The {\em fundamental quasisymmetric function} $L_{T}$ is given
by
$$
L_{T} = \sum_{T \subseteq S} M_{S} .
$$
Note that one can write $F(P)$
in terms of fundamental quasisymmetric functions:
$$
 F(P)
  =
 \sum_{S \subseteq [n-1]} h_{S} \cdot L_{S} .
$$

If the poset $P$ is a Boolean algebra $B_{n}$
then $h_{S} = \beta_{n}(S)$.
This is straightforward to observe using the classical
$R$-labeling of the Boolean algebra~\cite[Section~3.13]{Stanley}
or by direct enumeration~\cite[Corollary~3.12.2]{Stanley}.
The multiplicative property $B_{\ell+m} \cong B_\ell \times B_m$ allows
us to compute $F(B_n)$ modulo 2:

\begin{lemma}
\label{lemma_basic}
For $m = 2^{j}$ we have $F(B_m) \equiv M_{(m)} \bmod 2$.
Consequently,
for $n = 2^{j_1} + \cdots + 2^{j_k}$
with
$j_1 > \cdots > j_k \geq 0$ we have
$F(B_{n}) \equiv \prod_{i=1}^{k} M_{(2^{j_i})} \bmod 2$.
\end{lemma}
\begin{proof_special}
It is enough to prove the first statement. Recall
the congruence
$(a + b)^{2^{j}} \equiv a^{2^{j}} + b^{2^{j}} \bmod 2$.
Now we have
$$ \hspace*{20mm}
          F(B_{2^j}) 
   =      (w_{1} + w_{2} + \cdots)^{2^{j}}
   \equiv w_{1}^{2^{j}} + w_{2}^{2^{j}} + \cdots
   =      M_{(2^{j})}  \bmod 2 . 
   \hspace*{20mm} \qed $$
\end{proof_special}

An {\em ordered partition} $\pi$
of a set $[k]$
is a list of non-empty
pairwise disjoint sets
$(B_{1}, B_{2}, \ldots, B_{j})$
such that their union is $[k]$.

\begin{theorem}
For positive integers $m_{1}, m_{2}, \ldots, m_{k}$,
the product of the quasisymmetric functions
$M_{(m_{1})} \cdot M_{(m_{2})} \cdots M_{(m_{k})}$
is given by
$$ M_{(m_{1})} \cdot M_{(m_{2})} \cdots M_{(m_{k})}
  =
\sum_{\pi}
   M_{S(\pi)}  $$
where 
$$
S(\pi) = D^{-1}\left(\left(\sum_{i \in B_{1}} m_{i},
                \sum_{i \in B_{2}} m_{i},
                     \ldots,
                \sum_{i \in B_{j}} m_{i}\right)\right)  ,
$$
and $\pi = (B_{1}, B_{2}, \ldots, B_{j})$
ranges over all ordered partitions of the
set $[k]$.
\label{theorem_product}
\end{theorem}
\begin{proof}
The theorem can be proved by iterating Lemma~3.3
in~\cite{Ehrenborg}.
\end{proof}

In view of Lemma~\ref{lemma_basic},
the following is a restatement of Theorem~\ref{theorem_main} in the language
of quasisymmetric functions:

\begin{theorem}
The proportion of odd coefficients in the quasisymmetric function
$f = M_{(2^{j_1})} \cdot M_{(2^{j_2})} \cdots M_{(2^{j_k})}$
when expressed in the $L$-basis
only depends on $k$.
\label{theorem_knapsack}
\end{theorem}

\begin{proof}
There is a natural partially ordered set $\Pi_k$
on ordered partitions of $[k]$
with the cover relation $\pi \prec \sigma$
whenever $\sigma$ is obtained from $\pi$ by merging two
adjacent blocks. Expressing the quasisymmetric function $f$
in terms of the $L$-basis, we get
\begin{eqnarray*}
f
 & = &
\sum_{\pi\in\Pi_k} M_{S(\pi)} \\
 & = &
\sum_{\pi\in\Pi_k} \sum_{T \supseteq S(\pi)} (-1)^{|T-S(\pi)|} \cdot L_{T}  ,
\end{eqnarray*}
where $S(\pi)$ is defined as in Theorem~\ref{theorem_product}
for $m_i = 2^{j_i}$.
Thus the coefficient of $L_T$ is given by the sum
$$
\sum_{\pi\in\Pi_k\ :\ S(\pi)\subseteq T} (-1)^{|T-S(\pi)|}.
$$
As in the proof of Theorem~\ref{theorem_main},
considering pairs of sets $T$ and $T\cup \{i\}$, where $i\not\in T$ is
a nonessential element, we conclude that
the proportion of odd coefficients in the expansion of $f$ in the $L$-basis
only depends on sets $T$ consisting solely of essential elements.
The coefficients corresponding to such $T$ depend only on the poset
$\Pi_k$, that is, only on $k$, because of the above expression for
the coefficient of $L_T$.
\end{proof}

Theorem~\ref{theorem_main} implies that for $n = 2^{j}$
all the entries in the descent set statistics are odd.
Hence it is interesting to look at this data modulo $4$.
\begin{theorem}
For $n = 2^{j} \geq 4$ exactly half
of the descent set statistics
are congruent to $1$ modulo $4$,
and the other half
are congruent to $3$ modulo~$4$.
\label{theorem_2_power}
\end{theorem}
\begin{proof}
First
we claim that $F(B_{n}) \equiv M_{(n)} + 2 M_{(n/2,n/2)} \bmod 4$.
This identity follows from the observation
$(a + 2 b)^{2} \equiv a^{2} \bmod 4$
and by induction on $j$, where the induction step is
$$
F(B_{2 n})  \equiv  (M_{(n)} + 2 M_{(n/2,n/2)})^{2}
          \equiv  M_{(n)}^{2}
          \equiv  M_{(2 n)} + 2 M_{(n,n)} \bmod 4.
$$
We have the expansion
\begin{eqnarray*}
M_{(n)} + 2 M_{(n/2,n/2)}
 & = &
M_{\varnothing} + 2 M_{\{n/2\}} \\
 & = &
\sum_{S} (-1)^{|S|} \cdot L_{S}
 -
2 \cdot
\sum_{n/2 \in S} (-1)^{|S|} \cdot L_{S}  \\
 & = &
\sum_{n/2 \not\in S} (-1)^{|S|} \cdot L_{S}
 -
\sum_{n/2 \in S} (-1)^{|S|} \cdot L_{S}  .
\end{eqnarray*}
Hence the descent set statistics modulo $4$ are given by
$\beta_{n}(S) \equiv (-1)^{|S - \{n/2\}|} \bmod 4$.
Thus for $1 \not\in S$
the values
of $\beta_{n}(S)$ and $\beta_{n}(S \cup \{1\})$
have the opposite sign modulo~$4$, proving the result.
\end{proof}

\section{The signed descent set statistics}
\label{section_signed}

A {\em signed permutation} of size $n$ is
of the form $\pi = \pi_{1} \cdots \pi_{n}$
where each $\pi_{i}$ belongs to the set $\{\pm 1, \ldots, \pm n\}$
and $|\pi_{1}| \cdots |\pi_{n}|$ is a permutation.
Let ${\mathfrak S}^{\pm}_{n}$ be the set
of signed permutations of size $n$.
For ease of notation put $\pi_{0} = 0$.
The descent set of a signed permutation $\pi$
is a subset of $[n]$
defined as $\{i \: : \: \pi_{i-1} > \pi_{i}\}$.
For $S \subseteq [n]$
let~$\beta_{n}^{\pm}(S)$ denote the number of permutations
in~${\mathfrak S}^{\pm}_{n}$ with descent set $S$.

An equivalent way to using quasisymmetric functions to encode the flag
$f$-vector data of a poset is via the $\ab$-index.
Let $\av$ and $\bv$ be two non-commutative variables.
For $S \subseteq [n-1]$
let $u_{S}$ be the monomial $u_{1} u_{2} \cdots u_{n-1}$
where $u_{i} = \av$ if $i \not\in S$
and $u_{i} = \bv$ if $i \in S$.
The {\em $\ab$-index} of a poset~$P$ of rank $n$
is defined as the sum
$$    \Psi(P) = \sum_{S} h_{S} \cdot u_{S}   .  $$
When the poset $P$ is Eulerian
then its $\ab$-index can be written in
terms of $\cv = \av + \bv$
and $\dv = \av\bv + \bv\av$.
This more compact form
removes all linear redundancies
among the flag vector entries~\cite{Bayer_Klapper}.
The linear relations satisfied by  the flag $f$-vectors
of Eulerian posets are known as the
generalized Dehn-Sommerville relations~\cite{Bayer_Billera}.
Similarly to quasisymmetric functions, the
$\ab$-index and $\cd$-index also have an underlying
coalgebra structure.
For more details, see~\cite{Ehrenborg_Readdy}.

The poset associated to signed permutations is
the cubical lattice $C_{n}$,
that is, the face lattice of an $n$-dimensional cube.
Observe that $C_{n}$ has rank $n+1$.
We have
$$  \Psi(C_{n})
    =
 \sum_{S\subseteq [n]} \beta_{n}^{\pm}(S) \cdot u_{S} .  $$

A more general setting for the $\cd$-index of the cube is that of
zonotopes. Recall that a {\em zonotope} is a Minkowski sum of line segments.
Associated to every zonotope~$Z$ there is a central hyperplane
arrangement ${\mathcal H}$. Let $L$ be the intersection lattice of the
arrangement ${\mathcal H}$. A result by Billera, Ehrenborg,
and Readdy~\cite{Billera_Ehrenborg_Readdy_om}
shows how to compute the $\cd$-index of the zonotope from the $\ab$-index of
the intersection lattice $L$.
First, introduce 
the linear map 
$\omega$ from $\zab$ to $\zcd$ defined
on an $\ab$-monomial as follows.
Replace each occurrence of $\av \bv$ by $2 \dv$
and then replace the remaining letters by $\cv$.
The main result in~\cite{Billera_Ehrenborg_Readdy_om}
states that the $\cd$-index of the zonotope $Z$ is given by
\begin{equation}
 \Psi(Z)
     =
 \omega( \av \cdot \Psi(L) )   .
\label{equation_zonotope}
\end{equation}
In particular, for the cubical lattice we have
\begin{equation}
 \Psi(C_{n})
     =
 \omega( \av \cdot \Psi(B_{n}) )   ,
\label{equation_cubical}
\end{equation}
since the associated hyperplane arrangement is the coordinate
arrangement and its intersection lattice is the Boolean algebra.

Considering
Equation~(\ref{equation_zonotope})
modulo $2$, we observe
that $\Psi(Z) \equiv \cv^{n} \bmod 2$ and hence
we obtain the following result.
\begin{lemma}
All the entries of the flag $h$-vector
of a zonotope are odd.
In particular,
all the signed descent set statistics are odd.
\label{lemma_zonotope_odd}
\end{lemma}

In order to understand the flag $h$-vector modulo~$4$,
we need a few lemmas.
\begin{lemma}
After expanding the $\cd$-polynomial
$$   \sum_{i=0}^{n-2} \cv^{i} \cdot \dv \cdot \cv^{n-i-2} $$
into an $\ab$-polynomial, exactly half of the coefficients are odd.
\label{lemma_e}
\end{lemma}
\begin{proof}
Since $\dv \equiv \av\bv - \bv\av \bmod 2$,
it is sufficient to consider the identity
$$
    \sum_{i=0}^{n-2} \cv^{i} \cdot (\av\bv - \bv\av) \cdot \cv^{n-i-2}
 =
    \av \cdot (\av+\bv)^{n-2} \cdot \bv
   -
    \bv \cdot (\av+\bv)^{n-2} \cdot \av  .
$$
This identity holds
since the coefficient of an $\ab$-polynomial in the sum is
the number of occurrences of $\av\bv$
minus
the number of occurrences of $\bv\av$ in the monomial.
This difference only depends on the first and last letter in
the monomial and the identity follows.
To complete the proof, observe that 
out of $2^n$ $\ab$-monomials of degree $n$,
exactly
$2^{n-1}$ appear in the right-hand side of the identity.
\end{proof}

\begin{lemma}
Let $z$ and $w$ be two homogeneous polynomials in $\zab$
of degree $m$ and $n$, respectively,
each having exactly half
of their coefficients odd.
Then the two $\ab$-polynomials
$$    \cv^{i} \cdot z \cdot \cv^{j}
 \:\:\:\: \mbox{ and } \:\:\:\:
 z \cdot \cv^{n} + \cv^{m} \cdot w    $$
also each have exactly half of their coefficients odd.
\label{lemma_f}
\end{lemma}
\begin{proof}
We only prove the second statement
of the lemma.
We omit the proof of the first, as it is similar
and easier.
Let $u$ and $v$ be two $\ab$-monomials
of degrees $m$ and~$n$, respectively.
The coefficient of
$u \cdot v$
in $z \cdot \cv^{n} + \cv^{m} \cdot w$
is given by the sum
of the coefficients of $u$ in $z$
and of $v$ in $w$.
Hence the coefficient of $u \cdot v$ is even
when
the coefficients of $u$ and $v$
are both even ($2^{m-1} \cdot 2^{n-1}$ cases)
or
the coefficients of $u$ and $v$
are both odd ($2^{m-1} \cdot 2^{n-1}$ cases).
\end{proof}

Combining Lemmas~\ref{lemma_e} and~\ref{lemma_f}, we have:
\begin{proposition}
Let $\alpha_{1}, \ldots, \alpha_{n}$ be integers,
not all of which are even.
When the $\cd$-polynomial
$$   \sum_{i=0}^{n-2} \alpha_{i} \cdot \cv^{i} \cdot \dv \cdot \cv^{n-i-2} $$
is expanded into an $\ab$-polynomial,
exactly half of the coefficients are odd.
\label{proposition_d}
\end{proposition}

\begin{theorem}
For a zonotope $Z$ either
(i)
exactly half
of the flag $h$-vector entries
are congruent to $1$ modulo $4$,
and the other half
are congruent to $3$ modulo $4$;
or
(ii)
all the flag $h$-vector entries
are congruent to $1$ modulo~$4$.
\label{theorem_zonotope}
\end{theorem}
\begin{proof}
Considering the identity~(\ref{equation_zonotope}),
we observe that the only terms in the right-hand side with
non-zero coefficients modulo $4$
are $\cv^{n}$ and
those $\cd$-monomials
having
exactly one~$\dv$, that is,
$$
 \Psi(Z)
     \equiv
 \cv^{n}
     +
 2 \cdot \left(
     \sum_{i=0}^{n-2} \alpha_{i} \cdot \cv^{i} \cdot \dv \cdot \cv^{n-i-2}
         \right)
 \bmod 4 .
$$
If all the $\alpha_{i}$'s are even
then the flag $h$-vector entries
are congruent to $1$ modulo~$4$.
If at least one $\alpha_{i}$ is odd,
then by
Proposition~\ref{proposition_d}
exactly half of the flag $h$-vector entries
are congruent to $1$ modulo $4$
and the other half
are congruent to $3$ modulo~$4$.
\end{proof}

Now we consider the cubical lattice, that is,
the signed descent set statistics modulo $4$.
\begin{theorem}
For an integer $n \geq 2$, exactly half
of the signed descent set statistics
are congruent to $1$ modulo $4$,
and the other half
are congruent to $3$ modulo $4$.
\label{theorem_signed_permutations}
\end{theorem}
\begin{proof}
Observe that there are $2^{n}$ atoms in the cubical lattice.
Hence $h_{\{1\}} = 2^{n} - 1 \equiv 3 \bmod 4$.
Thus the result follows from Theorem~\ref{theorem_zonotope}.
\end{proof}

\section{Descent set statistics modulo $2p$}
\label{section_twice_prime}

For a set $S$ of integers and a non-zero integer $q$,
define the two notions
$q \cdot S$ and $S/q$ by
\begin{eqnarray*}
q \cdot S & = & \{q \cdot s \:\: : \:\: s \in S\} , \\
S/q       & = & \{s/q \:\: : \:\: s \in S \mbox{ \rm and } q \: | \: s \} .
\end{eqnarray*}
Recall that $\alpha_{n}(S)$ (resp.,\ $\beta_{n}(S)$)
denotes the number of permutations
in ${\mathfrak S}_{n}$ with descent set contained in
(resp.,\ equal to) $S$.

\begin{proposition}
Let $q = p^{t}$,
where $p$ is a prime and $t$ is a non-negative integer.
Let $n = r \cdot q$, where $r$ is a positive integer.
Then the descent set statistics modulo~$p$
are given by
$$    \beta_{n}(S)
  \equiv
       (-1)^{|S - q \cdot [r-1]|}
     \cdot
       \beta_{r}(S/q)
  \bmod p  ,
$$
where $S \subseteq [n-1]$.
\label{proposition_r_p_t}
\end{proposition}
\begin{proof}
Observe that
$$    M_{(1)}^{r}
  =
     F(B_{r})
  =
     \sum_{S \subseteq [r-1]} \alpha_{r}(S) \cdot M_{S}  . $$
In this quasisymmetric function identity make the
substitution $w_{i} \longmapsto w_{i}^{q}$. We then obtain
\begin{eqnarray*}
M_{(q)}^{r}
 & = &
     \sum_{S \subseteq [r-1]} \alpha_{r}(S) \cdot M_{q \cdot S}  \\
 & = &
     \sum_{S \subseteq [r-1]}
     \sum_{\onethingatopanother{T \subseteq [n-1]}
                               {q \cdot S \subseteq T}}
         \alpha_{r}(S) \cdot (-1)^{|T - q \cdot S|} \cdot L_{T}  \\
 & = &
     \sum_{T \subseteq [n-1]}
     \sum_{q \cdot S \subseteq T}
         \alpha_{r}(S) \cdot (-1)^{|T - q \cdot S|} \cdot L_{T}  \\
 & = &
     \sum_{T \subseteq [n-1]}
         (-1)^{|T - q \cdot [r-1]|}
       \cdot
         \left(
               \sum_{S \subseteq T/q}
                      \alpha_{r}(S) \cdot (-1)^{|T/q - S|}
         \right)
       \cdot
         L_{T}  \\
 & = &
     \sum_{T \subseteq [n-1]}
         (-1)^{|T - q \cdot [r-1]|}
       \cdot
         \beta_{r}(T/q)
       \cdot
         L_{T}  .
\end{eqnarray*}
Since $M_{1}^{q} \equiv M_{(q)} \bmod p$,
we have $F(B_{n}) = (M_{1}^{q})^{r} \equiv M_{(q)}^{r} \bmod p$.
Now by reading off the coefficients of $L_{T}$, the result follows.
\end{proof}

\begin{corollary}
Let $q = p^{t}$,
where $p$ is a prime and $t$ is a non-negative integer.
Then
$$ \beta_{q}(S) \equiv (-1)^{|S|} \bmod p  . $$
\end{corollary}
\begin{proof}
The claim can be deduced from Proposition~\ref{proposition_r_p_t}
by setting $r=1$.
A direct argument proceeds 
as follows.
Since $(a+b)^{q} \equiv a^{q} + b^{q} \bmod p$,
we have 
\begin{eqnarray*}
     F(B_{q})
  & = &
     (w_{1} + w_{2} + \cdots)^{q} \\
  & \equiv &
     w_{1}^{q} + w_{2}^{q} + \cdots \\
  & = &
     M_{(q)} = \sum_{S} (-1)^{|S|} \cdot L_{S} \bmod p.
\end{eqnarray*}
\end{proof}

\begin{corollary}
Let $q = p^{t}$,
where $p$ is a prime and $t$ is a non-negative integer.
Then
$$ \beta_{2q}(S) \equiv (-1)^{|S-\{q\}|} \bmod p  . $$
\label{corollary_2q_mod_p}
\end{corollary}
\begin{proof}
Follows from Proposition~\ref{proposition_r_p_t}
by setting $r=2$ and noting that
$\beta_{2}(\varnothing) = \beta_{2}(\{1\}) = 1$.
\end{proof}

For $n$ having the binary expansion
$n = 2^{j_{1}} + 2^{j_{2}} + \cdots + 2^{j_{k}}$,
where
$j_{1} > j_{2} > \cdots > j_{k} \geq 0$,
recall that an element $j \in [n-1]$ is
nonessential if $j$ is not a sum of a subset 
of
$\{2^{j_{1}}, 2^{j_{2}}, \ldots, 2^{j_{k}}\}$.

\begin{theorem}
Let $q = p^{t}$,
for $p$ an odd prime and $t$ a non-negative integer,
and let $n = r \cdot q$, where $r$ is a positive integer.
Suppose that there is a nonessential element
$j \in [n-1]$ that is not divisible by $q$.
Furthermore, suppose that there exist integers $a$ and $b$ not
divisible by $p$ such that $a\equiv b \bmod 2$, and
$\beta_{n}(S)$ is congruent to either $a$ or $b$ modulo $p$
for all $S\subseteq [n-1]$.
Then
\begin{eqnarray*}
 |\{S \subseteq [n-1] \: : \:
           \beta_{n}(S) \equiv a \bmod 2p \}|
 & = &
 |\{S \subseteq [n-1] \: : \:
           \beta_{n}(S) \equiv b \bmod 2p \}|      , \\
 |\{S \subseteq [n-1] \: : \:
           \beta_{n}(S) \equiv a+p \bmod 2p \}|
 & = &
 |\{S \subseteq [n-1] \: : \:
           \beta_{n}(S) \equiv b+p \bmod 2p \}|      ,
\end{eqnarray*}
In the case when the proportion $\rho(n)$  is~$1/2$,
the four cardinalities above are all equal to
$2^{n-3}$.
\label{theorem_prime_power_r}
\end{theorem}
\begin{proof}
Consider the collection of sets $S \subseteq [n-1]$
such that $\beta_{n}(S) \equiv a \equiv b \bmod 2$.
For~$S$ in this collection such that $j \not\in S$,
we have 
$\beta_{n}(S) \equiv \beta_{n}(S \cup \{j\}) \bmod 2$.
However, since $q$ does not divide~$j$,
we have 
$\beta_{n}(S) \equiv -\beta_{n}(S \cup \{j\}) \bmod p$
by Proposition~\ref{proposition_r_p_t}, that is,
$\beta_{n}(S) \not\equiv \beta_{n}(S \cup \{j\}) \bmod p$,
as $a$ (resp., $b$) is not congruent to $-a$
(resp.,\ $-b$) modulo $p$.
Hence 
this collection 
splits into two classes of equal size
when divided according to the value of
$\beta(S)$ modulo $2p$.
The same argument holds for
the sets $S$ satisfying
$\beta(S) \equiv a+p \equiv b+p \bmod 2$.
\end{proof}

By the Chinese remainder theorem, we have 
$\beta(S) \equiv \pm 1, p \pm 1 \bmod 2 p$.
For non-Mersenne primes we can say more.
(Recall that $\rho(n)$ is defined as the ratio of
the number of subsets $S\subseteq [n-1]$ such that
$\beta_{n}(S)$ is odd to the total number $2^{n-1}$ of subsets
of $[n-1]$.)
\begin{theorem}
Let $q = p^{t}$
be an odd prime power which has $k$ $1$'s in its binary expansion.
Suppose that $q > 2^{k} - 1$, that is, $q$ is not a Mersenne prime.
Then
\begin{eqnarray*}
 |\{S \subseteq [q-1] \: : \:
           \beta_{q}(S) \equiv 1 \bmod 2p \}|
 & = &
 |\{S \subseteq [q-1] \: : \:
           \beta_{q}(S) \equiv -1 \bmod 2p \}|      , \\
 |\{S \subseteq [q-1] \: : \:
           \beta_{q}(S) \equiv p-1 \bmod 2p \}|
 & = &
 |\{S \subseteq [q-1] \: : \:
           \beta_{q}(S) \equiv p+1 \bmod 2p \}|     ,
\end{eqnarray*}
In the case the proportion $\rho(q)$ is  $1/2$,
the four cardinalities above are equal to
$2^{q-3}$.
\label{theorem_prime_a}
\end{theorem}
\begin{proof}
Since $q > 2^{k} - 1$,
as in the proof of Theorem~\ref{theorem_knapsack}
there exists a nonessential element $j \in [q-1]$.
Thus Theorem~\ref{theorem_prime_power_r} applies.
\end{proof}

When $q = p$ is a prime
and $k=2$, Theorem~\ref{theorem_prime_a}
applies only to the Fermat primes
which are greater than $3$,
that is, $5$, $17$, $257$ and $65537$.
We also know that the proportion is $1/2$ for
the case $k=3$, that is, primes whose binary expansion
has three $1$'s.
The first few such primes are
$7, 11, 13, 19, 37, 41, 67, 73, 97$;
see sequence A081091 in
The On-Line Encyclopedia of Integer Sequences.

For prime powers of the
form $q = p^t$ with $t \geq 2$,
the only case with $k=2$ we know is $q = 3^{2}$.
Similarly, with $k = 3$ we know six cases:
$5^{2}$,  $7^{2}$,  $3^{4}$,  $17^{2}$,  $23^{2}$,
$257^{2}$ and $65537^{2}$.
It is not surprising that the squares of the Fermat's primes
and the square of $3^{2}$ appear in this list.
The two sporadic cases are $7^{2}$ and $23^{2}$.

The next theorem concerns permutations of size
twice a prime power.

\begin{theorem}
Let $q = p^{t}$
be an odd prime power which has $k$ $1$'s in its binary expansion.
Then
\begin{eqnarray*}
 |\{S \subseteq [2q-1] \: : \:
           \beta_{2q}(S) \equiv 1 \bmod 2p \}|
 & = &
 |\{S \subseteq [2q-1] \: : \:
           \beta_{2q}(S) \equiv -1 \bmod 2p \}|, \\
 |\{S \subseteq [2q-1] \: : \:
           \beta_{2q}(S) \equiv p-1 \bmod 2p \}|
 & = &
 |\{S \subseteq [2q-1] \: : \:
           \beta_{2q}(S) \equiv p+1 \bmod 2p \}|,
\end{eqnarray*}
In the case  the proportion $\rho(q)$ is $1/2$,
the four cardinalities above are equal to
$2^{2q-3}$.
\label{theorem_prime2}
\end{theorem}
\begin{proof}
By Corollary~\ref{corollary_2q_mod_p},
$\beta_{2q}(S) \equiv (-1)^{|S-\{q\}|} \bmod p$.
Furthermore, since $2q$ is even,
the element $1$ is nonessential,
and Theorem~\ref{theorem_prime_power_r} applies.
\end{proof}

\section{The descent set polynomial}

A different approach to view the results 
from the previous sections is in terms of the 
{\em descent set polynomial} 
$$ 
     Q_{n}(t) = \sum_{S \subseteq [n-1]} t^{\beta_{n}(S)}.
$$
The degree of this polynomial is
the $n$th Euler number $E_{n}$.  For $n \geq 2$ the polynomial is
divisible by $2t$.  Theorems~\ref{theorem_main},
\ref{theorem_2_power}, \ref{theorem_prime_a} 
and~\ref{theorem_prime2} can be reformulated as follows.
\begin{theorem}
\begin{itemize}
\item[(i)]
For a positive integer $n$ we have
$Q_{n}(-1) = 2^{n} \cdot (1/2 - \rho(n))$.
In particular, when $n$ has two or three $1$'s in its binary expansion,
then $-1$ is a root of $Q_{n}(t)$.
\item[(ii)]
For $n = 2^{j} \geq 4$
the imaginary unit $\imagi$ is a root of $Q_{n}(t)$.
\item[(iii)]
Let $q = p^{t}$ be a prime power, where $p$ is an odd prime.
Suppose that $q$ has $k$ $1$'s in its binary expansion
and satisfies $q > 2^{k} - 1$.
Let $\zeta$ be a primitive $2p$-th root of unity.
Then
$$      Q_{q}(\zeta)
    =
       2^{q} \cdot \Rere(\zeta) \cdot \left(\rho(q) - \frac{1}{2}\right)  , $$
where $\Rere(\zeta)$ denotes the real part of $\zeta$.
\item[(iv)]
Let $q = p^{t}$ be a prime power, where $p$ is an odd prime.
Suppose that $q$ has $k$ $1$'s in its binary expansion.
Let $\zeta$ be a primitive $2p$-th root of unity.
Then
$$     Q_{2q}(\zeta)
    =
      2^{2q} \cdot \Rere(\zeta) \cdot \left(\rho(q) - \frac{1}{2}\right)  .$$
\end{itemize}
\label{theorem_Q}
\end{theorem}
It is curious to observe that the polynomial $Q_{n}(t)$
quite often has
zeroes occurring at roots of unity.
An equivalent formulation is that
$Q_{n}(t)$ often has
cyclotomic polynomials
$\Phi_{k}(t)$
as factors.
(Recall that the {\em cyclotomic polynomial}
$\Phi_{k}(t)$ is defined as the product
$\prod_{\zeta} (t - \zeta)$,
where $\zeta$ ranges over all primitive $k$th roots of unity.)
See Table~\ref{table_P}
for the cyclotomic factors of $Q_{n}(t)$ for $n \leq 23$.

\begin{lemma}
Let $q$ be an odd prime power. Then the cyclotomic polynomial
$\Phi_{q}$ does not divide the descent set polynomial $Q_{n}(t)$.
\end{lemma}
\begin{proof}
If $q$ is a power of an odd prime $p$, then $\Phi_q(1)=p$.
Since $Q_{n}(1) = 2^{n-1}$ has no odd factors, the lemma follows.
\end{proof}

\begin{lemma}
Let $q$ be the odd prime power $p^{t}$. Then
\begin{itemize}
\item[(i)]
If $n = 2^{j}$ then the cyclotomic polynomial
$\Phi_{2 q}$ does not divide $Q_{n}(t)$.
\item[(ii)]
If $n$ has four $1$'s in its binary expansion
and $p \geq 5$ then the cyclotomic polynomial
$\Phi_{2 q}$ does not divide $Q_{n}(t)$.
\item[(iii)]
If $n$ has five $1$'s in its binary expansion
and $p \geq 11$ then the cyclotomic polynomial
$\Phi_{2 q}$ does not divide $Q_{n}(t)$.
\end{itemize}
\end{lemma}
\begin{proof}
We have $\Phi_{2q}(-1)=p$.
Since $Q_{n}(-1) = 2^{n} \cdot (1/2 - \rho(n))$
the result follows by consulting Table~\ref{table_one}.
\end{proof}

\section{Quadratic factors in the descent set polynomial}
\label{section_quadratic}

In order to study the double root behavior of the
descent set polynomial
$Q_{n}(t)$ or, equivalently, quadratic factors in $Q_{n}(t)$,
we need to prove a few identities for the descent set statistics.
We begin  by introducing the multivariate $\ab$- and $\cd$-indexes.
Let $\av_{1}, \av_{2}, \ldots$
and $\bv_{1}, \bv_{2}, \ldots$ be non-commutative variables.
For $S \subseteq [n-1]$
let $u_{S}$ be the monomial $u_{1} u_{2} \cdots u_{n-1}$
where $u_{i} = \av_{i}$ if $i \not\in S$
and $u_{i} = \bv_{i}$ if $i \in S$.
The {\em multivariate $\ab$-index} of a poset~$P$ of rank $n$
is defined as the sum
$$    \Psi(P) = \sum_{S} h_{S} \cdot u_{S},     $$
where $S$ ranges over all subsets of $[n-1]$.

\begin{lemma}\label{multivariate_cd}
For an Eulerian poset $P$ the multivariate $\ab$-index can
be written in terms of the non-commutative
variables
$\cv_{i} = \av_{i} + \bv_{i}$
and $\dv_{i,i+1} = \av_{i}\bv_{i+1} + \bv_{i}\av_{i+1}$.
\end{lemma}

\begin{proof}
Observe that by adding the index $i$ to the $i$th letter in an $\ab$-monomial
of degree $n-1$, we obtain a natural bijection between the regular
and the multivariate $\ab$-indices of the same poset $P$. Thus the statement
of the lemma is equivalent to the statement that the regular $\ab$-index of
$P$ can be expressed in terms of the variables $\cv = \av+\bv$ and $\dv
= \av\bv + \bv\av$.
\end{proof}

\noindent
In this case,
we call the resulting polynomial
the
{\em multivariate $\cd$-index}.
Observe that for a rank~$n$ Eulerian poset
each of the indices $1$ through $n$ appears in
each monomial of the multivariate $\cd$-index.

\begin{proposition}
Let $h_{S}$ be the flag $h$-vector of an Eulerian poset $P$ of rank~$n$,
or more generally, $h_{S}$ belongs to
the generalized Dehn-Sommerville subspace.
Let $T \subseteq [n-1]$ such that
$T$ contains an interval
$[s,t] = \{s, s+1, \ldots, t\}$
of odd cardinality
with $s-1, t+1 \not\in T$.
Then 
$$ \sum_{S\subseteq [n-1]}
     (-1)^{|S \cap T|} \cdot h_{S}
  =
    0    .  $$
\label{proposition_cd}
\end{proposition}
\begin{proof}
The sum is obtained from the multivariate $\ab$-index
of the poset $P$ by setting
$\av_{i} = 1$ and
$$  \bv_{i} = \left\{
                \begin{array}{r l}
                   -1 & \mbox{ if } i \in T, \\
                    1 & \mbox{ otherwise. }
                \end{array}
             \right.    $$
Notice that
$\cv_{i} = 0$ for $i \in [s,t]$
and
that $\dv_{s-1,s} = \dv_{t,t+1} = 0$.
If $s=1$ we set $\dv_{0,1} = 0$,
and if $t=n-1$ we set $\dv_{n-1,n} = 0$.
Since $P$ is Eulerian,
the multivariate $\ab$-index can be written
in terms of multivariate $\cd$-monomials (Lemma~\ref{multivariate_cd}).
A multivariate $\cd$-monomial that contains
$\dv_{s-1,s}$ or $\dv_{t,t+1}$ evaluates to zero.
Since the interval $[s,t]$ has odd size,
a multivariate $\cd$-monomial not containing
$\dv_{s-1,s}$ and $\dv_{t,t+1}$
must
contain at least one variable $\cv_{i}$
with $i \in [s,t]$.
Hence this monomial also evaluates to zero.
\end{proof}

Observe that the identity in Proposition~\ref{proposition_cd}
is a part of the generalized Dehn-Sommerville relations;
see~\cite{Bayer_Billera}.

\begin{theorem}
If the binary expansion of $n$ has two $1$'s and $n > 3$, then
$\Phi_{2}^{2}$ divides~$Q_{n}(t)$.
\end{theorem}
\begin{proof}
Suppose that $n = m_{1} + m_{2}$, where
$m_{1} = 2^{j_{1}}$,
$m_{2} = 2^{j_{2}}$, and $j_{1} > j_{2}$.
From the proof of Theorem~\ref{theorem_knapsack} we have
$$ \beta_{n}(S)
    \equiv
  \left\{
      \begin{array}{c c l}
         1 & {} \bmod 2 & \mbox{ if } |S \cap \{m_{1}, m_{2}\}| = 0, 2 , \\
         0 & {} \bmod 2 & \mbox{ if } |S \cap \{m_{1}, m_{2}\}| = 1 .
      \end{array}
  \right.
$$
Hence
\begin{eqnarray*}
   Q_{n}^{\prime}(-1)
 & = &
   \sum_{S\subseteq [n-1]} \beta_{n}(S) \cdot (-1)^{\beta_{n}(S)-1}  \\
 & = &
   \sum_{S\subseteq [n-1]}
         (-1)^{|S \cap \{m_{1}, m_{2}\}|}
       \cdot
         \beta_{n}(S),
\end{eqnarray*}
which is zero by
Proposition~\ref{proposition_cd}.
\end{proof}

\begin{theorem}
If $n = 2^{j} \geq 4$ then
$\Phi_{4}^{2}$ divides $Q_{n}(t)$.
\label{theorem_2_power_square}
\end{theorem}
\begin{proof}
Let $m = n/2$.
The proof of Theorem~\ref{theorem_2_power} states that
$\beta(S) \equiv (-1)^{|S-\{m\}|} \bmod 4$.
Let
$\imagi$ be the imaginary unit, so that
$\imagi^2 = -1$.
Observe that $\imagi^{(-1)^{k}-1} = (-1)^{k}$.
We have
\begin{eqnarray*}
   Q_{n}^{\prime}(\imagi)
 & = &
   \sum_{S\subseteq [n-1]} \beta_{n}(S) \cdot \imagi^{\beta_{n}(S)-1}  \\
 & = &
   \sum_{S\subseteq [n-1]}
       (-1)^{|S-\{m\}|} \cdot \beta_{n}(S).
\end{eqnarray*}
By
Proposition~\ref{proposition_cd},
$Q_{n}^{\prime}(\imagi) = 0$,
since $S-\{m\} = S \cap \{1, \ldots, m-1,m+1, \ldots, n-1\}$
and $m-1$ is odd.
\end{proof}

The next result applies to
prime powers that have two $1$'s in their binary expansion.
The only cases known so far
are the five known Fermat primes
$3, 5, 17, 257, 65537$ and the prime power $3^{2}$.
\begin{theorem}
Let $q = p^{t}$ be a prime power, where $p$ is an odd prime
and assume that $q$ has two $1$'s in its binary expansion.
Then the cyclotomic polynomial $\Phi_{2p}^{2}$ divides~$Q_{2q}(t)$.
\label{theorem_Fermat}
\end{theorem}
\begin{proof}
In this case $n = 2 q = m + 2$, where $m = 2^{j}$.
From the proof of Theorem~\ref{theorem_knapsack} we have
$$ \beta_{2q}(S)
    \equiv
  \left\{
      \begin{array}{c c l}
         1 & {} \bmod 2 & \mbox{ if } |S \cap \{2,m\}| = 0,2 , \\
         0 & {} \bmod 2 & \mbox{ if } |S \cap \{2,m\}| = 1 .
      \end{array}
  \right.
$$
Hence combining it with the proof of
Corollary~\ref{corollary_2q_mod_p},
we have
$$ \beta_{2q}(S)
    \equiv
  \left\{
      \begin{array}{r c l}
         (-1)^{|S - \{q\}|}
                 & {} \bmod 2p
                 & \mbox{ if } |S \cap \{2,m\}| = 0,2 , \\
         p + (-1)^{|S - \{q\}|}
                 & {} \bmod 2p
                 & \mbox{ if } |S \cap \{2,m\}| = 1 .
      \end{array}
  \right.
$$
Thus for
$\zeta = \Rere(\zeta) + \Imim(\zeta) \cdot \imagi$
a $2p$-th primitive root of unity, we have that
\begin{eqnarray*}
    \zeta^{\beta_{2q}(S)}
 & = &
    (-1)^{|S \cap \{2,m\}|} \cdot \zeta^{(-1)^{|S - \{q\}|}} \\
 & = &
    (-1)^{|S \cap \{2,m\}|}
  \cdot
    \left(\Rere(\zeta)
           +
          (-1)^{|S - \{q\}|} \cdot \Imim(\zeta) \cdot \imagi \right) .
\end{eqnarray*}
Evaluating the sum
and using the fact that
$$|S \cap \{2,m\}| + |S - \{q\}| \equiv |S \cap \{2,q,m\}| \bmod 2$$
we have
\begin{eqnarray*}
   \zeta \cdot Q_{2q}^{\prime}(\zeta)
 & = &
   \sum_{S\subseteq [2q-1]} \beta_{2q}(S) \cdot \zeta^{\beta_{2q}(S)}  \\
 & = &
   \Rere(\zeta)
      \cdot
   \sum_{S\subseteq [2q-1]}
      (-1)^{|S \cap \{2,m\}|}
        \cdot
      \beta_{2q}(S) \\
 & & +\ 
   \Imim(\zeta)
      \cdot
   \imagi
      \cdot
   \sum_{S\subseteq [2q-1]}
      (-1)^{|S \cap \{2,q,m\}|}
        \cdot
      \beta_{2q}(S) ,
\end{eqnarray*}
where both sums vanish by
Proposition~\ref{proposition_cd}.
\end{proof}

\section{The signed descent set polynomial}
\label{section_signed_descent}

Similarly to the descent set polynomial
we can define the {\em signed descent set polynomial:}
$$ Q^{\pm}_{n}(t) = \sum_{S\subseteq [n]} t^{\beta_{n}^{\pm}(S)}  ; $$
see Section~\ref{section_signed}
for definitions relevant to signed permutations.
The degree of this polynomial is the $n$th
signed Euler number~$E^{\pm}_{n}$, which is the number
of alternating signed permutations of size $n$.
Yet again, for $n \geq 1$ this polynomial is divisible by $2t$.
Theorem~\ref{theorem_signed_permutations}
can now be stated as follows.

\begin{theorem}
For $n \geq 2$
the signed descent set polynomial $Q^{\pm}_{n}(t)$
has the cyclotomic factor~$\Phi_{4}$.
\end{theorem}

The space of quasisymmetric functions of type $B$ is defined
as $\BQSym = \Zzz[s] \tensor \QSym$.
Quasisymmetric functions of type $B$ were first defined
by Chow~\cite{Chow}.
We will view them to be functions in the variables
$s, w_{1}, w_{2}, \ldots$, that are quasisymmetric
in $w_{1}, w_{2}, \ldots$.
For a composition
$(\gamma_{0}, \ldots, \gamma_{m})$
define the monomial quasisymmetric function of type $B$
by
$$    M^{B}_{(\gamma_{0}, \gamma_{1},\ldots,\gamma_{m})}
  =
     s^{\gamma_{0}-1}
  \cdot
     M_{(\gamma_{1},\ldots,\gamma_{m})}    .   $$

A third method to encode the flag vector data of a poset $P$
of rank at least $1$
is the {\em quasisymmetric function of type $B$}
\begin{equation}
    F_{B}(P)
  =
    \sum_{c} M^{B}_{\rho(c)}  ,
\label{equation_type_B}
\end{equation}
where the sum is over all chains
$c = \{\hz = x_{0} < x_{1} < \cdots < x_{m} = \ho\}$
in the poset $P$; see~\cite{Ehrenborg_Readdy_Tchebyshev}.
A different way to write
equation~(\ref{equation_type_B}) is
\begin{equation}
    F_{B}(P)
  =
    \sum_{\hz < x \leq \ho} s^{\rho(x)-1} \cdot F([x,\ho])  .
\label{equation_type_B_version_2}
\end{equation}
The diamond product of two posets $P$ and $Q$ is
$$P \diamond Q = (P - \{\hz\}) \times (Q - \{\hz\}) \cup \{\hz\}.$$
Using identity~(\ref{equation_type_B_version_2}) one can show that
the type $B$ quasisymmetric function of a poset is
multiplicative with respect to the diamond product of posets,
that is,
$F_{B}(P \diamond Q) = F_{B}(P) \cdot F_{B}(Q)$.
Applying the bijection~$D$ between compositions
and subsets, we have 
$$
F_{B}(P) = \sum_{S \subseteq [n-1]} f_{S} \cdot M^{B}_{S},
$$
where we write $M^{B}_{S}$ instead of $M^{B}_{D(S)}$,
and the poset $P$ has rank $n$.
The {\em fundamental quasisymmetric function of type $B$,}
denoted by $L^{B}_{T}$, is given
by
$$
L^{B}_{T} = \sum_{T \subseteq S} M^{B}_{S} .
$$
Then the flag $h$-vector appears as the coefficients
in the decomposition
$$
 F_{B}(P)
  =
 \sum_{S \subseteq [n-1]} h_{S} \cdot L^{B}_{S} ,
$$
where the poset $P$ has rank $n$.

The cubical lattice $C_{n}$ has
rank $n+1$ and can obtained as a diamond power of the
Boolean algebra $B_{2}$, that is,
$C_{n} = B_{2}^{\diamond n}$.
Therefore we have the following result.

\begin{lemma}
The type $B$ quasisymmetric function
of the cubical lattice is given by
$$F_{B}(C_{n}) = (s + 2 \cdot M_{(1)})^{n}.$$
\end{lemma}

\begin{theorem}
For $p$ an odd prime 
the cyclotomic polynomial $\Phi_{4p}$
divides the signed descent set polynomial $Q^{\pm}_{p}(t)$.
\end{theorem}
\begin{proof}
Observe that modulo $4$ we have 
\begin{eqnarray*}
F_{B}(C_{p})
 & \equiv &
(s + 2 \cdot M_{(1)})^{p} \\
 & \equiv &
s^{p} + 2 \cdot p \cdot s^{p-1} \cdot M_{(1)} \\
 & \equiv &
M_{(p+1)} + 2 \cdot p \cdot M_{(p,1)}  \\
 & \equiv &
M_{\varnothing} + 2 \cdot M_{\{p\}}  \\
 & \equiv &
\sum_{S\subseteq [p]} (-1)^{|S|} \cdot L^{B}_{S}
 +
2 \cdot \sum_{p \in S} (-1)^{|S|-1} \cdot L^{B}_{S} \\
 & \equiv &
\sum_{p \not\in S} (-1)^{|S|} \cdot L^{B}_{S}
 +
\sum_{p \in S} (-1)^{|S|-1} \cdot L^{B}_{S}  \bmod 4 .
\end{eqnarray*}
Hence the signed descent set statistics satisfy
$\beta^{\pm}_{p}(S) \equiv (-1)^{|S-\{p\}|} \bmod 4$
for $S \subseteq [p]$.
Now modulo $p$ we have
\begin{eqnarray*}
F_{B}(C_{p})
 & \equiv &
(s + 2 \cdot M_{(1)})^{p} \\
 & \equiv &
s^{p} + 2 \cdot M_{(p)} \\
 & \equiv &
M^{B}_{(p+1)} + 2 \cdot M^{B}_{(1,p)} \\
 & \equiv &
M_{\varnothing} + 2 \cdot M_{\{1\}}  \bmod p.
\end{eqnarray*}
This directly implies that
$\beta^{\pm}_{p}(S) \equiv (-1)^{|S-\{1\}|} \bmod p$.
Combining these two statements we obtain
\begin{equation}
       \beta^{\pm}_{p}(S)
  \equiv
       \left\{ \begin{array}{c c l}
              (-1)^{|S|}        & {} \bmod 4p &
 \mbox{ \rm if } 1,p \not\in S, \\
              (-1)^{|S|-1}      & {} \bmod 4p &
 \mbox{ \rm if } 1,p \in S, \\
   2 \cdot p + (-1)^{|S|}       & {} \bmod 4p &
 \mbox{ \rm if } 1 \not\in S, p \in S, \\
   2 \cdot p + (-1)^{|S|-1}     & {} \bmod 4p &
 \mbox{ \rm if } 1 \in S, p \not\in S. \\
               \end{array} \right.
\label{equation_signed_modulo_4_p}
\end{equation}
Observe that for $p \not\in S$ we have 
$\beta^{\pm}_{p}(S) \equiv \beta^{\pm}_{p}(S \cup \{p\}) + 2 \cdot p \bmod 4p$,
implying that
$\zeta^{\beta^{\pm}_{p}(S)} = - \zeta^{\beta^{\pm}_{p}(S \cup \{p\})}$
for $\zeta$ a $4p$-th primitive root of unity.
Now sum over all subsets of~$[p]$, and the result follows.
\end{proof}

\begin{theorem}
For $p$ an odd prime, 
$\Phi_{4p}^2$ does not divide the signed descent set polynomial
$Q^{\pm}_{p}(t)$. In fact,
evaluating the derivative of
the signed descent set polynomial
$Q^{\pm}_{p}(t)$ at $\zeta$,
where
$\zeta$ is a $4p$-th primitive root of unity,
gives
$$
   \zeta \cdot {Q^{\pm}_{p}}^{\prime}(\zeta)
 =
   \Imim(\zeta) \cdot \imagi
     \cdot
   (-1)^{(p-1)/2} \cdot 2^{p} \cdot p \cdot E_{p-1}
.
$$
\end{theorem}
\begin{proof}
From~(\ref{equation_signed_modulo_4_p})
we have:
\begin{eqnarray*}
     \zeta \cdot {Q^{\pm}_{p}}^{\prime}(\zeta)
 & = &
     \sum_{S\subseteq [p]} \beta^{\pm}_{p}(S) \cdot \zeta^{\beta^{\pm}_{p}(S)} \\
 & = &
     \sum_{S\subseteq [p]}
            \beta^{\pm}_{p}(S)
          \cdot
            (-1)^{|S \cap \{1,p\}|}
          \cdot
            \zeta^{(-1)^{|S-\{1\}|}} \\
 & = &
     \Rere(\zeta)
   \cdot
     \sum_{S\subseteq [p]}
            \beta^{\pm}_{p}(S)
          \cdot
            (-1)^{|S \cap \{1,p\}|} \\
 &   &
+
     \Imim(\zeta)
  \cdot
     \imagi
  \cdot
     \sum_{S\subseteq [p]}
            \beta^{\pm}_{p}(S)
          \cdot
            (-1)^{|S \cap \{1,p\}|}
          \cdot
            (-1)^{|S-\{1\}|} .
\end{eqnarray*}
The first sum is zero by
Proposition~\ref{proposition_cd}.
The second sum simplifies to
$$
     \Imim(\zeta)
  \cdot
     \imagi
  \cdot
     \sum_{S\subseteq [p]}
            \beta^{\pm}_{p}(S)
          \cdot
            (-1)^{|S \cap [1,p-1]|}  .
$$
This sum can evaluated by setting
$\av_{j} = 1$, $\bv_{1} = \cdots = \bv_{p-1} = -1$
and $\bv_{p} = 1$ in the
the multivariate $\ab$-index of the cubical lattice $C_{p}$.
Observe that $\cv_{1} = \cdots = \cv_{p-1} = 0$
and $\dv_{p-1,p} = 0$. Hence the only surviving $\cd$-monomial
is $\dv_{1,2} \cdots \dv_{p-2,p-1} \cv_{p}$.
The coefficient of this monomial is computed as follows:
\begin{eqnarray*}
 \left[\dv^{(p-1)/2} \cv\right] \Psi(C_{p})
 & = &
2^{(p-1)/2}
 \cdot
 \left[(2\dv)^{(p-1)/2} \cv\right] \Psi(C_{p}) \\
 & = &
2^{(p-1)/2}
 \cdot
\left(
 \left[(\av\bv)^{(p-1)/2} \av\right] \av \cdot \Psi(B_{p}) 
\right. \\
 &   &
\left.
   +
 \left[(\av\bv)^{(p-1)/2} \bv\right] \av \cdot \Psi(B_{p})
\right) \\
 & = &
2^{(p-1)/2}
 \cdot
p
 \cdot
 \left[\bv (\av\bv)^{(p-3)/2}\right] \Psi(B_{p-1})  \\
 & = &
2^{(p-1)/2}
 \cdot
p
 \cdot
E_{p-1}  .
\end{eqnarray*}
The third step is
MacMahon's ``Multiplication Theorem'';
see~\cite[Article~159]{MacMahon}.
It can be stated in terms of the $\ab$-indices as follows:
$$   [u \av v] \Psi(B_{m+n})  +  [u \bv v] \Psi(B_{m+n})
   =
     \binomial{m+n}{m} \cdot 
     [u] \Psi(B_{m})  +  [v] \Psi(B_{n})  ,  $$
where $u$ and $v$ have degrees $m-1$ and $n-1$, respectively.
The monomial itself evaluates to $(-2)^{(p-1)/2} \cdot 2$,
since
$\dv_{1,2} = \cdots = \dv_{p-2,p-1} = -2$ and $\cv_{p} = 2$.
Combining all the factors, the evaluation at $\zeta$ follows.
\end{proof}

\newcommand{\bl}[1]{\mathbf{{\textcolor{blue}{#1}}}}
\begin{table}[ht!]
$$
\begin{array}{r r l}
 n & \mbox{\rm degree} & \mbox{\rm cyclotomic factors of $Q_{n}(t)$} \\ \hline
 3 &            2 & \bl{\Phi_{2}} \\
 4 &            5 & \bl{\Phi_{4}^{2}} \\
 5 &           16 & \bl{\Phi_{2}^{2}} \cdot \bl{\Phi_{10}} \\
 6 &           61 & \bl{\Phi_{2}^{2}} \cdot \bl{\Phi_{6}^{2}} \cdot \Phi_{10} \\
 7 &          272 & \bl{\Phi_{2}} \\
 8 &         1385 & \bl{\Phi_{4}^{2}} \cdot \Phi_{28} \\
 9 &         7936 & \bl{\Phi_{2}^{2}} \cdot \bl{\Phi_{6}} \cdot \Phi_{18} \\
10 &        50521 & \bl{\Phi_{2}^{2}} \cdot \Phi_{6} \cdot \bl{\Phi_{10}^{2}} \cdot \Phi_{18} \cdot \Phi_{30} \\
11 &       353792 & \bl{\Phi_{2}} \cdot \Phi_{6} \cdot \bl{\Phi_{22}} \\
12 &      2702765
  & \bl{\Phi_{2}^{2}} \cdot \Phi_{6} \cdot \Phi_{10}
                  \cdot \Phi_{18} \cdot \Phi_{22} \cdot \Phi_{22}
                  \cdot \Phi_{66} \cdot \Phi_{110}
                  \cdot \Phi_{198} \\
13 &     22368256
  & \bl{\Phi_{2}} \cdot \bl{\Phi_{26}} \\
14 &    1.993 \cdot 10^{8}
  & \bl{\Phi_{2}} \cdot \Phi_{2} \cdot \Phi_{4} \cdot \bl{\Phi_{14}} \cdot \Phi_{14}
                          \cdot \Phi_{26} \cdot \Phi_{28}
                          \cdot \Phi_{182} \\
15 &   1.904 \cdot 10^{9}
  & - \\
16 &  1.939 \cdot 10^{10}
  &  \bl{\Phi_{4}^{2}} \cdot \Phi_{12} \cdot
                          \Phi_{20} \cdot \Phi_{44} \cdot
                          \Phi_{52} \cdot \Phi_{60} \cdot
                          \Phi_{156} \cdot \Phi_{220} \cdot
                          \Phi_{260} \cdot \Phi_{572}        \\
17 & 2.099 \cdot 10^{11}
  &  \bl{\Phi_{2}^{2}} \cdot \bl{\Phi_{34}}
\\
18 & 2.405 \cdot 10^{12}
  &
\bl{\Phi_{2}^2} \cdot
\bl{\Phi_{6}^2} \cdot
\Phi_{18} \cdot
\Phi_{34} \cdot
\Phi_{102} \cdot
\Phi_{306} \\
19 & 2.909 \cdot 10^{13}
  &
\bl{\Phi_{2}} \cdot
\bl{\Phi_{38}} \\
20 & 3.704 \cdot 10^{14}
  &
\bl{\Phi_{2}^2} \cdot
\Phi_{6} \cdot
\Phi_{10} \cdot
\Phi_{30} \cdot
\Phi_{34} \cdot
\Phi_{38}^2 \cdot
\Phi_{102} \cdot
\Phi_{114} \cdot
\Phi_{170} \\
& &
\cdot
\Phi_{190} \cdot
\Phi_{510} \cdot
\Phi_{570} \cdot
\Phi_{646} \cdot
\Phi_{1938} \cdot
\Phi_{3230} \cdot
\Phi_{9690} \\
21 & 4.951 \cdot 10^{15}
  &
\bl{\Phi_{2}} \cdot
\Phi_{6} \cdot
\Phi_{14} \cdot
\Phi_{42} \\
22 & 6.935 \cdot 10^{16}
  &
\bl{\Phi_{2}} \cdot
\Phi_{2} \cdot
\Phi_{14} \cdot
\bl{\Phi_{22}} \cdot
\Phi_{22} \cdot
\Phi_{154} \\
23 & 1.015 \cdot 10^{18}
  &
-
\end{array}
$$
\caption{Cyclotomic factors of $Q_{n}(t)$.}
\label{table_P}
\end{table}

\begin{table}[ht!]
$$
\begin{array}{r r l}
 n & \mbox{\rm degree} &
\mbox{\rm cyclotomic factors of $Q^{\pm}_{n}(t)$} \\ \hline
2 & 3
&
\bl{\Phi_{4}} \\
3 & 11
&
\bl{\Phi_{4}} \cdot
\Phi_{8} \cdot
\bl{\Phi_{12}} \\
4 & 57
&
\bl{\Phi_{4}} \cdot
\Phi_{16} \cdot
\Phi_{32} \\
5 & 361
&
\bl{\Phi_{4}} \cdot
\Phi_{16} \cdot
\bl{\Phi_{20}} \cdot
\Phi_{32} \cdot
\Phi_{80} \\
6 & 2763
&
\bl{\Phi_{4}} \cdot
\Phi_{8} \cdot
\Phi_{24} \cdot
\Phi_{32} \cdot
\Phi_{40} \cdot
\Phi_{96} \cdot
\Phi_{120} \cdot
\Phi_{160} \\
7 & 24611
&
\bl{\Phi_{4}} \cdot
\Phi_{8} \cdot
\Phi_{24} \cdot
\bl{\Phi_{28}} \cdot
\Phi_{32} \cdot
\Phi_{56} \cdot
\Phi_{168} \cdot
\Phi_{224} \\
8 & 250737
&
\bl{\Phi_{4}} \cdot
\Phi_{32} \cdot
\Phi_{64} \cdot
\Phi_{224} \cdot
\Phi_{448} \cdot
\Phi_{512} \\
9 & 2873041
&
\bl{\Phi_{4}} \cdot
\Phi_{12} \cdot
\Phi_{32} \cdot
\Phi_{36} \cdot
\Phi_{64} \cdot
\Phi_{96} \cdot
\Phi_{192} \cdot
\Phi_{288} \cdot
\Phi_{448} \\
&& \cdot
\Phi_{512}^{2} \cdot
\Phi_{576} \cdot
\Phi_{1344} \cdot
\Phi_{1536} \cdot
\Phi_{4032} \cdot
\Phi_{4608} \\
10 & 36581523
&
\bl{\Phi_{4}} \cdot
\Phi_{8} \cdot
\Phi_{24} \cdot
\Phi_{40} \cdot
\Phi_{64} \cdot
\Phi_{72} \cdot
\Phi_{120} \cdot
\Phi_{192} \cdot
\Phi_{320} \cdot
\Phi_{360} \\
&& \cdot
\Phi_{448} \cdot
\Phi_{512} \cdot
\Phi_{960} \cdot
\Phi_{1344} \cdot
\Phi_{1536} \cdot
\Phi_{2240} \cdot
\Phi_{2560} \cdot
\Phi_{6720} \cdot
\Phi_{7680} \\
11 & 5.123 \cdot 10^{8}
&
\bl{\Phi_{4}} \cdot
\Phi_{8} \cdot
\Phi_{40} \cdot
\bl{\Phi_{44}} \cdot
\Phi_{64} \cdot
\Phi_{88} \cdot
\Phi_{192} \cdot
\Phi_{320} \cdot
\Phi_{440} \\
&& \cdot
\Phi_{512} \cdot
\Phi_{704} \cdot
\Phi_{960} \cdot
\Phi_{2112} \cdot
\Phi_{2560} \cdot
\Phi_{3520} \cdot
\Phi_{5632} \\
12 & 7.828 \cdot 10^{9}
&
\bl{\Phi_{4}} \cdot
\Phi_{16} \cdot
\Phi_{32} \cdot
\Phi_{48} \cdot
\Phi_{96} \cdot
\Phi_{160} \cdot
\Phi_{176} \cdot
\Phi_{288} \cdot
\Phi_{352} \\
&& \cdot
\Phi_{480} \cdot
\Phi_{512} \cdot
\Phi_{528} \cdot
\Phi_{1056} \cdot
\Phi_{1440} \cdot
\Phi_{1536} \cdot
\Phi_{1760} \cdot
\Phi_{2560} \\
&& \cdot
\Phi_{3168} \cdot
\Phi_{4608} \cdot
\Phi_{5280} \cdot
\Phi_{5632} \\
13 & 1.296 \cdot 10^{11}
&
\bl{\Phi_{4}} \cdot
\Phi_{16} \cdot
\Phi_{32} \cdot
\Phi_{48} \cdot
\bl{\Phi_{52}} \cdot
\Phi_{160} \cdot
\Phi_{208} \cdot
\Phi_{352} \cdot
\Phi_{416} \\
&& \cdot
\Phi_{512}^{2} \cdot
\Phi_{624} \cdot
\Phi_{1536} \cdot
\Phi_{1760} \cdot
\Phi_{2080} \cdot
\Phi_{4576} \cdot
\Phi_{5632} \cdot
\Phi_{6656} \\
14 & 2.310 \cdot 10^{12}
&
\bl{\Phi_{4}} \cdot
\Phi_{8} \cdot
\Phi_{32} \cdot
\Phi_{56} \cdot
\Phi_{104} \cdot
\Phi_{224} \cdot
\Phi_{352} \cdot
\Phi_{416} \cdot
\Phi_{512} \\
&& \cdot
\Phi_{728} \cdot
\Phi_{1536} \cdot
\Phi_{2464} \cdot
\Phi_{2912} \cdot
\Phi_{3584} \cdot
\Phi_{4576} \cdot
\Phi_{5632} \cdot
\Phi_{6656} \\
15 & 4.411 \cdot 10^{13}
&
\bl{\Phi_{4}} \cdot
\Phi_{8} \cdot
\Phi_{12} \cdot
\Phi_{20} \cdot
\Phi_{24} \cdot
\Phi_{32} \cdot
\Phi_{40} \cdot
\Phi_{56} \cdot
\Phi_{60} \\
&& \cdot
\Phi_{96} \cdot
\Phi_{120} \cdot
\Phi_{160} \cdot
\Phi_{168} \cdot
\Phi_{224} \cdot
\Phi_{280} \cdot
\Phi_{416} \cdot
\Phi_{480} \cdot
\Phi_{512} \\
&& \cdot
\Phi_{672} \cdot
\Phi_{840} \cdot
\Phi_{1120} \cdot
\Phi_{1248} \cdot
\Phi_{1536} \cdot
\Phi_{2080} \cdot
\Phi_{2560} \cdot
\Phi_{2912} \cdot
\Phi_{3360} \\
&& \cdot
\Phi_{5632} \cdot
\Phi_{6240} \cdot
\Phi_{6656} \cdot
\Phi_{7680} \cdot
\Phi_{8736} \\
16 & 8.986 \cdot 10^{14}
&
\bl{\Phi_{4}} \cdot
\Phi_{64} \cdot
\Phi_{128} \cdot
\Phi_{192} \cdot
\Phi_{320} \cdot
\Phi_{640} \cdot
\Phi_{896} \cdot
\Phi_{960} \\
&& \cdot
\Phi_{1024} \cdot
\Phi_{1664} \cdot
\Phi_{3072} \cdot
\Phi_{4480} \cdot
\Phi_{5120} \cdot
\Phi_{8320} \\
17 & 1.945 \cdot 10^{16}
&
\bl{\Phi_{4}} \cdot
\Phi_{64} \cdot
\bl{\Phi_{68}} \cdot
\Phi_{128} \cdot
\Phi_{640} \cdot
\Phi_{896} \\
&& \cdot
\Phi_{1024}^{2} \cdot
\Phi_{1088} \cdot
\Phi_{2176} \cdot
\Phi_{4480} \cdot
\Phi_{5120} \\
18 & 4.458 \cdot 10^{17}
&
\bl{\Phi_{4}} \cdot
\Phi_{8} \cdot
\Phi_{24} \cdot
\Phi_{72} \cdot
\Phi_{128} \cdot
\Phi_{136} \cdot
\Phi_{384} \cdot
\Phi_{408} \cdot
\Phi_{640} \\
&& \cdot
\Phi_{1024} \cdot
\Phi_{1152} \cdot
\Phi_{1224} \cdot
\Phi_{1920} \cdot
\Phi_{2176} \cdot
\Phi_{3072} \cdot
\Phi_{5760} \cdot
\Phi_{6528} \cdot
\Phi_{9216}
\end{array}
$$
\caption{Cyclotomic factors of $Q^{\pm}_{n}(t)$.}
\label{table_P_pm}
\end{table}

\section{Concluding remarks}

Is there a reason why $\rho(n) - 1/2$ factors so nicely?
See Table~\ref{table_one}.

The two main results
for unsigned permutations
in Section~\ref{section_proportion},
Theorems~\ref{theorem_main}
and~\ref{theorem_knapsack},
can also be proved using the
$\ab$-index and the mixing operator;
see~\cite{Ehrenborg_Fox}.
We have omitted this approach since Kummer's theorem and
the quasisymmetric functions
are more succinct in this case.

Tables~\ref{table_P} and~\ref{table_P_pm} contain cyclotomic factors
of polynomials $Q_n(t)$ and $Q^{\pm}_n(t)$ for small $n$. Those factors
whose presence is explained in this paper are highlighted in boldface.
Here are several observations
about the data in Table~\ref{table_P}:
\begin{itemize}
\item[(i)] 
All the indices $k$ of cyclotomic factors $\Phi_k$ of the
polynomials $Q_n(t)$ are even.

\item[(ii)] 
Any prime factor $p$ that occurs in an index of a cyclotomic
factor of $Q_{n}(t)$ is less than or equal to $n$.

\item[(iii)] 
If $\Phi_{k_{1}}$ and $\Phi_{k_{2}}$ are factors
of $Q_{n}(t)$, so is $\Phi_{\gcd(k_{1},k_{2})}$. That is, the set
of indices is closed under the meet operation in the divisor lattice.

\item[(iv)] 
If $k_{1}$ divides $k_{2}$,
$k_{2}$ divides $k_{3}$ and $\Phi_{k_{1}}$ and $\Phi_{k_{3}}$
occur as factors in $Q_{n}(t)$, then so does $\Phi_{k_{2}}$.
This is convexity in the divisor lattice.

\item[(v)] 
If both $\Phi_{k_{1}}$ and $\Phi_{k_{2}}$
divide $Q_{n}(t)$, where $k_{1}$ divides $k_{2}$,
then the multiplicity of $\Phi_{k_{1}}$ is greater than or equal to the
multiplicity of $\Phi_{k_{2}}$.

\item[(vi)] 
If $p$ is not a Mersenne prime
then the largest cyclotomic factor occurring in $Q_{p}(t)$ is $\Phi_{2p}$.

\item[(vii)] 
When $\rho(n) \neq 1/2$ then there are no cyclotomic factors
in the descent set polynomial $Q_{n}(t)$.

\item[(viii)] 
For all primes $p$ we conjecture that $\Phi_{2p}^2$ divides $Q_{2p}$.
\end{itemize}
Moreover for the signed descent set polynomial we observe that:
\begin{itemize}
\item[(ix)]
For $n \geq 3$ the
the cyclotomic polynomial $\Phi_{4n}$
divides the signed descent set polynomial $Q^{\pm}_{n}(t)$.

\item[(x)]
For $n \geq 5$ the
the cyclotomic polynomial $\Phi_{4n(n-1)}$
divides the signed descent set polynomial $Q^{\pm}_{n}(t)$.
\end{itemize}
Can these phenomena be explained?

For what pairs of an integer $n$ and a prime number $p$
does the descent set statistic~$\beta_{n}(S)$ only
take two values modulo $p$?

Finally, we end with two number-theoretic questions.
Are there infinitely many primes
whose binary expansion
has three $1$'s?
The only reference for these primes we found
is The On-Line Encyclopedia of Integer Sequences,
sequence A081091.
Are there any more prime powers with two or three ones in
its binary expansion?

\section*{Acknowledgements}

The authors thank the referee for
improving the proof of Theorem~\ref{theorem_main}.
The authors also thank
the MIT Mathematics Department where
this research was carried out.
The second author was partially
supported by
National Security Agency grant H98230-06-1-0072,
and the third author was partially
supported by
National Science Foundation grant DMS-0604423.

\newcommand{\journal}[6]{{\sc #1,} #2, {\it #3} {\bf #4} (#5), #6.}
\newcommand{\book}[4]{{\sc #1,} ``#2,'' #3, #4.}
\newcommand{\bookf}[5]{{\sc #1,} ``#2,'' #3, #4, #5.}
\newcommand{\thesis}[4]{{\sc #1,} ``#2,'' Doctoral dissertation, #3, #4.}
\newcommand{\springer}[4]{{\sc #1,} ``#2,'' Lecture Notes in Math.,
                       Vol.\ #3, Springer-Verlag, Berlin, #4.}
\newcommand{\preprint}[3]{{\sc #1,} #2, preprint #3.}
\newcommand{\preparation}[2]{{\sc #1,} #2, in preparation.}
\newcommand{\appear}[3]{{\sc #1,} #2, to appear in {\it #3}}
\newcommand{\submitted}[4]{{\sc #1,} #2, submitted to {\it #3}, #4.}
\newcommand{\JCTA}{J.\ Combin.\ Theory Ser.\ A}
\newcommand{\AdvancesinMathematics}{Adv.\ Math.}
\newcommand{\JournalofAlgebraicCombinatorics}{J.\ Algebraic Combin.}

\newcommand{\communication}[1]{{\sc #1,} personal communication.}


{

}

\bigskip
\noindent
{\em D.\ Chebikin
     Department of Mathematics,
     MIT,
     Cambridge, MA~02139}, \\
{\em R.\ Ehrenborg
     Department of Mathematics,
     University of Kentucky,
     Lexington, KY~40506}, \\
{\em P.\ Pylyavskyy,
     Department of Mathematics,
     University of Michigan, 
     Ann Arbor, MI 48109}, \\
{\em M.\ Readdy,
     Department of Mathematics,
     University of Kentucky,
     Lexington, KY~40506}.

\end{document}